\documentclass[a4j,12pt]{article}  
\usepackage{geometry}
\usepackage[T1]{fontenc} 
\usepackage{times} 
\usepackage{amsmath,amssymb,bm, mathpazo}
\usepackage{amsfonts}
\usepackage{pifont}
\usepackage{mathrsfs}
\usepackage{graphicx}
\usepackage{amscd}

\makeatletter

\newtheorem{thm}{\theoremname}[section]
\newtheorem{prop}[thm]{\propositionname}
\newtheorem{cor}[thm]{\corollaryname}
\newtheorem{df}{\definitionname}[section]
\newtheorem{lemma}[thm]{\lemmaname}

\newcommand{\theoremname}{\underline{Theorem}}
\newcommand{\definitionname}{\underline{Definition}}
\newcommand{\lemmaname}{\underline{Lemma}}
\newcommand{\corollaryname}{\underline{Corollary}}
\newcommand{\axiomname}{\underline{Axiom}}
\newcommand{\propositionname}{\underline{Proposition}}
\newcommand{\problemname}{Problem}
\newcommand{\examplename}{\underline{Example}}
\newcommand{\remarkname}{\underline{Remark}}

\def\tedsymbol{\vcenter{\hbox{\vrule\@height.5em\@width.5em}}}
\def\ted{{\unskip\nobreak\hfil\penalty50
 \quad\hbox{}\nobreak\hfil \hbox{$\tedsymbol$}
 \parfillskip\z@ \finalhyphendemerits\z@\par}}

\providecommand{\abs}[1]{\lvert#1\rvert}

\makeatother

\title{On the geometric realization and subdivisions of dihedral sets}
\author{
\LARGE{Sho Saito}\footnote{Second year graduate student at Graduate School of Mathematics, Nagoya University, email: \texttt{m09019h@math.nagoya-u.ac.jp}}}
\begin{document}
\date{\empty}
\maketitle
  \nocite{Harada-01,Katura-05,Nakano-03,Takagi-96,Artin-74,Kuga-68,MP-64,Infeld-96,Borceux-01,Serre-95,Iyanaga-99,Iyanaga-02}
  \section{Introduction}
  By expressing the geometric realization of simplicial sets and cyclic sets as filtered colimits, Drinfeld \cite{d} 
  proved in a substantially simplified way the fundamental facts that geometric realization preserves finite limits, and that the group of orientation-preserving homeomorphisms of the interval $[0,1]$ (resp. the circle $\mathbb{R}/\mathbb{Z}$) acts on the realization of a simplicial (resp. cyclic) set. 
  In this paper, we first review Drinfeld's method and then introduce an analogous expression for the geometric realization of dihedral sets. 
  We also see how these expressions lead to a clarified description of 
  subdivisions of simplicial, cyclic, and dihedral sets. 
  \subsection{Terminology and results}
  Let $\Delta$ be the simplicial index category of the finite linearly ordered sets $[n]=\{0<\cdots<n\}$, $n\geq 0$, and order-preserving maps. 
  By definition, a {\bf simplicial set} is a contravariant functor $X[-]$ from $\Delta$ to the category $\operatorname{Sets}$ of sets. 
  We write $\Delta[n][-]$ for the standard simplicial $n$-simplex $\operatorname{Hom}_{\Delta}([-],[n])$ and set $\Delta[n]=\{(z_0,\ldots,z_n)\in[0,1]^{n+1}\mid\sum^n_{i=0}z_i=1\}$, with the standard euclidean topology. 
  \begin{df}[Milnor \cite{milnor}]
  The {\bf geometric realization} $\abs{X[-]}$ of the simplicial set $X[-]$ is the colimit $$\operatorname{colim}_{\Delta[n][-]\to X[-]}\Delta[n],$$ where the index category is formed by simplicial maps from standard simplicial simplices to $X[-]$ and natural transformations. 
  \end{df} 
  We can construct $\abs{X[-]}$ explicitly by $$\abs{X[-]}=\coprod_{n\geq 0}X[n]\times\Delta[n]/\sim, $$ where $X[n]$ is given the discrete topology and where $\sim$ is the equivalence relation generated by the relation that identifies $(x,\theta_\ast z)$ with $(\theta^\ast x,z)$ for every pair $(x,z)\in X[n]\times\Delta[m]$ and for every map $\theta:[m]\to[n]$ in $\Delta$. 
  However, we interchangably adopt any other space having the universal property of the colimit as a definition of $\abs{X[-]}$, since there are canonical isomorphisms between such spaces. 
  
  Drinfeld \cite{d} re-defined the geometric realization as a filtered colimit, showing that his definition is equivalent to Milnor's one. 
  To introduce his expression, we first need to extend the simplicial set $X[-]$ to a contravariant functor $\widetilde{X}[-]$ from the category $\Delta_{\operatorname{big}}$ of all finite linearly ordered sets. 
  This process is explained in detail in section 2. 
  We denote by $\mathcal{F}$ the set of all finite subsets of $[0,1]$, viewed as a category with morphisms being inclusions. 
  For each $F\in\mathcal{F}$, we order the set of connected components $\pi_0([0,1]\setminus F)$ by declaring that $[x]\leq [y]$ if $x\leq y\in[0,1]\setminus F$.  
  If $F\subset G$ there is an order-preserving map $\pi_0([0,1]\setminus G)\to\pi_0([0,1]\setminus F)$. 
  \begin{thm}[Drinfeld \cite{d}]
  \label{thm1.1}
  The geometric realization $\abs{X[-]}$ of the simplicial set $X[-]$, as a set, is given by the 
  colimit $$\operatorname{colim}_{F\in\mathcal{F}}\widetilde{X}[\pi_0([0,1]\setminus F)]. $$ 
  \end{thm}
  
  The point is that the index category $\mathcal{F}$ is a filtering, so that it becomes clear that geometric realization, as a functor from the category of simplicial sets to $\operatorname{Sets}$, preserves finite limits. 
  This expression also makes it obvious that the group $\operatorname{Homeo}([0,1],\partial[0,1])$ of order-preserving homeomorphisms of $[0,1]$ acts on the set $\abs{X[-]}$. 
  In fact, it gives a new proof of the following stronger statements, which were proved first by Gabriel-Zisman \cite{gz}. 
  Let $\mathcal{K}$ be the category of $k$-spaces (i.e., spaces with all compactly open subsets being open), and topologize the group $\operatorname{Homeo}([0,1],\partial[0,1])$ by the subspace topology in $\operatorname{Hom}_{\mathcal{K}}([0,1],[0,1])$ with the standard $k$-space topology. 
  Since $\Delta[n]$ is a $k$-space for each $n$, $\abs{X[-]}$ is also a $k$-space. 
  We say that an action by $\operatorname{Homeo}([0,1],\partial[0,1])$ on the realization is continuous if the map $\operatorname{Homeo}([0,1],\partial[0,1])\times\abs{X[-]}\to\abs{X[-]}$ is a continuous map of $k$-spaces. 
  \begin{cor}[Gabriel-Zisman \cite{gz}] 
  \label{cor1.2}
  \begin{enumerate}
  \item Geometric realization, considered as a functor to $\mathcal{K}$, preserves finite limits. 
  \item The group $\operatorname{Homeo}([0,1],\partial[0,1])$ acts continuously on the realization of a simplicial set. 
  \end{enumerate}
  \end{cor}
  
  A similar argument applies to cyclic sets introduced by Connes. 
  Let $\Delta C$ be the category that makes the family $\{C_{n+1}\}_{n\geq 0}$ of cyclic groups of order $n+1$ into a crossed simplicial group in the sense of Fiedorowicz-Loday. 
  Recall: 
  \begin{df}[\cite{loday}
  ]
  \label{crossed}
  A {\bf crossed simplicial group} is a family of groups $\{G_n\}_{n\geq 0}$ together with a category $\Delta G$ that has one object $[n]$ for each $n\geq 0$, containing $\Delta$ as a subcategory, and satisfies the following conditions: 
  \begin{enumerate}
  \item The group of automorphisms $\operatorname{Aut}_{\Delta G}[n]$ on each $[n]$ is isomorphic to the group $G_n^{\operatorname{op}}$. 
  \item Every morphism $[m]\to[n]$ can be uniquely written as a composite $\phi\circ g$ with $\phi\in\operatorname{Hom}_{\Delta}([m],[n])$ and $g\in\operatorname{Aut}_{\Delta G}[m]$. 
  \end{enumerate}  
  \end{df}  
  For a $\Delta G$-set $X[-]$ (i.e. a contravariant functor from $\Delta G$ to $\operatorname{Sets}$), we define its {\bf geometric realization} to be the realization of the underlying simplicial set $\Delta^{\operatorname{op}}\hookrightarrow(\Delta G)^{\operatorname{op}}\stackrel{X[-]}{\to}\operatorname{Sets}$. 

  {\bf Cyclic sets} are defined to be $\Delta C$-sets. 
  We construct $\Delta C$ as a category of $\mathbb{Z}_+$-categories after Drinfeld \cite{d}. 
  Here $\mathbb{Z}_+$ is the additive monoid of non-negative integers, and a {\bf $\mathbb{Z}_+$-category} is a category $\mathcal{C}$ together with a nontrivial monoid map $\mathbb{Z}_+\to\operatorname{End}_{\mathcal{C}}\operatorname{id}_{\mathcal{C}}$. 
  I.e., there is a non-identity endomorphism $1_c:c\to c$ on every object $c\in\operatorname{ob}\mathcal{C}$ such that $f\circ1_{c_1}=1_{c_2}\circ f$ for every $f:c_1\to c_2$. 
  A $\mathbb{Z}_+$-{\bf functor} (resp. {\bf isomorphism}) is a functor (resp. isomorphism) between $\mathbb{Z}_+$-categories that preserves the structural endomorphisms. 
  
  The most basic example of a $\mathbb{Z}_+$-category is the circle $\mathbb{R}/\mathbb{Z}$. 
  Morphisms from $x$ to $y$ are homotopy classes of continuous maps $f:[0,1]\to\mathbb{R}\to\mathbb{R}/\mathbb{Z}$ such that $f(0)=x$ and $f(1)=y$, with $[0,1]\to\mathbb{R}$ non-decreasing and $\mathbb{R}\to\mathbb{R}/\mathbb{Z}$ the canonical projection.  
  The $\mathbb{Z}_+$-category structure is given by $1_x=$(class of degree 1 loops based at $x$). 
  If $F\subset \mathbb{R}/\mathbb{Z}$ is a finite subset, the set of connected components $\pi_0(\mathbb{R}/\mathbb{Z}\setminus F)$ can be considered as a $\mathbb{Z}_+$-category. 
  The set of morphisms from $c$ to $d$ is defined by choosing representatives $x_c\in c$ and $x_d\in d$: $$\operatorname{Hom}_{\pi_0(\mathbb{R}/\mathbb{Z}\setminus F)}(c,d)=\operatorname{Hom}_{\mathbb{R}/\mathbb{Z}}(x_c,x_d). $$  
  If $F\subset G$ there is a $\mathbb{Z}_+$-functor $\pi_0(\mathbb{R}/\mathbb{Z}\setminus G)\to\pi_0(\mathbb{R}/\mathbb{Z}\setminus F)$. 

  There is a way of constructing a $\mathbb{Z}_+$-category $\mathcal{A}_{\operatorname{cyc}}$ from a given small category $\mathcal{A}$ (\cite{d}, Example 4 of section 2). 
  In the particular case where $\mathcal{A}=\{a_0<\cdots<a_n\}$ is a linearly ordered set, viewed as a category, then $\mathcal{A}_{\operatorname{cyc}}$ is the $\mathbb{Z}_+$-category that has the same objects as $\mathcal{A}$ and that has morphisms generated by those in $\mathcal{A}$ together with one new generator $a_n\to a_0$. 
  The structural endomorphisms are given by $1_{a_i}:a_i\to a_n\to a_0\to a_i$. 
  If $\mathcal{A}=[n]$ then $[n]_{\operatorname{cyc}}$ is identified with the full $\mathbb{Z}_+$-subcategory $$\{[0],[1/(n+1)],\ldots,[n/(n+1)]\}\subset\mathbb{R}/\mathbb{Z},$$ 
  where $[-]$ denotes the class in $\mathbb{R}/\mathbb{Z}$. 
  (For notational simplicity we will frequently omit such brackets.)   
  If $\mathcal{A}=\pi_0([0,1]\setminus F)$ with $F\in\mathcal{F}$ containing $0$ or $1$, then $\pi_0([0,1]\setminus F)_{\operatorname{cyc}}$ is identified with $\pi_0(\mathbb{R}/\mathbb{Z}\setminus\overline{F})$ where $\overline{F}=\{[x]\in\mathbb{R}/\mathbb{Z}\mid x\in F\}$.   
  We set $\Delta C$ (resp. $\Delta_{\text{big}}C$) to be the category of the $\mathbb{Z}_+$-categories $[n]_{\text{cyc}}$ (resp. small $\mathbb{Z}_+$-categories isomorphic to some $[n]_{\operatorname{cyc}}$) and $\mathbb{Z}_+$-functors.


  We likewise extend the cyclic set $X[-]$ to a contravariant functor $\widetilde{X}[-]$ from the extended category $\Delta_{\text{big}}C$, and write $\mathcal{F}^\prime$ for the set of all finite subsets of $\mathbb{R}/\mathbb{Z}$, viewed as a filtered category. 
  \begin{thm}[Drinfeld \cite{d}]
  \label{thm1.3}
  The geometric realization $\abs{X[-]}$ of the cyclic set $X[-]$, as a set, is given by the filtered colimit $$\operatorname{colim}_{F\in\mathcal{F}^{\prime}}\widetilde{X}[\pi_0(\mathbb{R}/\mathbb{Z}\setminus F)]. $$ 
  In particular, the group $\operatorname{Homeo}^+\mathbb{R}/\mathbb{Z}$ of orientation-preserving homeomorphisms of the circle $\mathbb{R}/\mathbb{Z}$ acts continuously on $\abs{X[-]}$. 
  \end{thm} 
  
  We define $\Delta D$ to be the category whose set of objects is the same as that of $\Delta C$, and whose set of morphisms from $[m]_{\operatorname{cyc}}$ to $[n]_{\operatorname{cyc}}$ is the disjoint union of the sets of covariant and contravariant $\mathbb{Z}_+$-functors from $[m]_{\text{cyc}}$ to $[n]_{\text{cyc}}$. 
  ({\bf Remarks}:  
  1. If $C$ is a $\mathbb{Z}_+$-category, the $\mathbb{Z}_+$-category structure on the opposite category $C^{\operatorname{op}}$ is given by $1_c=(1_c)^{\operatorname{op}}$ for $c\in\operatorname{ob}C^{\operatorname{op}}$. 
  2. Considering the \textit{disjoint union} means that \textit{a functor cannot be both covariant and contravariant}. 
  As a consequence, e.g. $\operatorname{Hom}_{\Delta D}([0]_{\text{cyc}},[0]_{\text{cyc}})$ has two elements: the identity $\operatorname{id}_{[0]_{\text{cyc}}}:[0]_{\text{cyc}}\to[0]_{\text{cyc}}$ as a covariant functor, and $\operatorname{id}_{[0]_{\text{cyc}}}:[0]_{\text{cyc}}\to[0]_{\text{cyc}}^{\operatorname{op}}$ as a contravariant functor.) 
  The composition in $\Delta D$ is defined by usual composition of functors, under the rule that the composite of two covariant or contravariant functors should be covariant, and the composite of covariant and contravariant functors should be contravariant. 
    
  Let $D_n$ denote the dihedral group of order $2n$. 
  We show that $\Delta D$ makes the family $\{D_{n+1}\}_{n\geq 0}$ into a crossed simplicial group. 
  A {\bf Dihedral set} is a $\Delta D$-set $X[-]$, and there is a similar extension to a contravariant functor $\widetilde{X}[-]$ from the extended category $\Delta_{\operatorname{big}}D$ of $\mathbb{Z}_+$-categories isomorphic to some $[n]_{\operatorname{cyc}}$ and covariant and contravariant $\mathbb{Z}_+$-functors. 
  It is known that the geometric realization of a dihedral set admits a continuous action by the orthogonal group $O(2)$. 
  We prove the following new, stronger result: 
  \begin{thm}
  \label{thm1.4}
  The geometric realization $\abs{X[-]}$ of the dihedral set $X[-]$, as a set, is given by the filtered colimit $$\operatorname{colim}_{F\in\mathcal{F}}\widetilde{X}[\pi_0(\mathbb{R}/\mathbb{Z}\setminus F)]. $$  
  In particular, the group $\operatorname{Homeo}\mathbb{R}/\mathbb{Z}$ of all homeomorphisms of the circle $\mathbb{R}/\mathbb{Z}$ acts continuously on $\abs{X[-]}$. 
  \end{thm}

  For every positive integer $r$, B\"{o}kstedt-Hsiang-Madsen \cite{bhm} defined an operation called the {\bf $r$-fold edgewise subdivision} of simplicial or cyclic sets. 
  We denote by $\operatorname{sd}_rX[-]$ the $r$-fold edgewise subdivision of the simplicial or cyclic set $X[-]$, whose definition is recalled in section 3. 
  Write $\mathcal{F}_r$ and $\mathcal{F}_r^\prime$ for the set of finite subsets of $[0,r]$ and $\mathbb{R}/r\mathbb{Z}$, respectively. 
  We introduce the following expression for subdivisions, which gives a new natural proof of the result by B\"{o}kstedt-Hsiang-Madsen \cite{bhm} that $\abs{X[-]}$ and $\abs{\operatorname{sd}_rX[-]}$ are canonically homeomorphic.   
  \begin{thm}
  \label{thm1.5}
  For the simplicial (resp. cyclic) set $X[-]$, $\abs{\operatorname{sd}_rX[-]}$ is given by the filtered colimit $$\operatorname{colim}_{F\in\mathcal{F}_r}\widetilde{X}[\pi_0([0,r]\setminus F)]$$ $$\text{(resp.} \operatorname{colim}_{F\in\mathcal{F}_r^\prime}\widetilde{X}[\pi_0(\mathbb{R}/r\mathbb{Z}\setminus F)]\text{)}$$
  and hence admits an action by $G_r=\operatorname{Homeo}([0,r],\partial[0,r])$ (resp. $G_r=\operatorname{Homeo}^+\mathbb{R}/r\mathbb{Z}$).  
  In particular, the bijection $[0,r]\to[0,1]$ (resp. $\mathbb{R}/r\mathbb{Z}\to\mathbb{R}/\mathbb{Z}$) given by $x\mapsto x/r$ induces isomorphisms $D_r:\abs{\operatorname{sd}_rX[-]}\to\abs{X[-]}$ and $d_r:G_r\to G=\operatorname{Homeo}([0,1],\partial[0,1])$ (resp. $d_r:G_r\to G=\operatorname{Homeo}^+\mathbb{R}/r\mathbb{Z}$) such that the 
  diagram
  \begin{displaymath}
  \begin{CD}
  G_r\times\abs{\operatorname{sd}_rX[-]} @>>> \abs{\operatorname{sd}_rX[-]}  \\ 
  @VV{d_r\times D_r}V @VV{D_r}V  \\
  G\times\abs{X[-]} @>>> \abs{X[-]} 
  \end{CD}
  \end{displaymath}
  commutes. 
  \end{thm}
    
  For dihedral sets, Spali\'{n}ski \cite{spalinski} defined two types of subdivision operations, $\operatorname{sd}_r$ and $\operatorname{sd}_r^{\operatorname{e}}$. 
  In his definition, both operations assign to a dihedral set $X[-]$ a simplicial set. 
  We re-define in section 3 $\operatorname{sd}_rX[-]$ and $\operatorname{sd}_r^{\operatorname{e}}X[-]$ to have the richer structures of a $\Delta_rD$-set and a $\Delta_{2r}D$-set, respectively.   
  Here $\Delta_rD$ is the category that makes $\{D_{r(n+1)}\}_{n\geq 0}$ into a crossed simplicial group.  
  \begin{thm}
  \label{thm1.6}
  For the dihedral set $X[-]$, $\abs{\operatorname{sd}_rX[-]}$ (resp. $\abs{\operatorname{sd}_r^{\operatorname{e}}X[-]}$) is given by the filtered colimit 
  $$\operatorname{colim}_{F\in\mathcal{F}_r^\prime}\widetilde{X}[\pi_0(\mathbb{R}/r\mathbb{Z}\setminus F)]$$ 
  $$\text{(resp. }\operatorname{colim}_{F\in\mathcal{F}_{2r}}\widetilde{X}[\pi_0(\mathbb{R}/2r\mathbb{Z}\setminus F)]\text{)}$$  
  and hence admits an action by $G_r=\operatorname{Homeo}\mathbb{R}/r\mathbb{Z}$ (resp. $G_r^{\operatorname{e}}=\operatorname{Homeo}\mathbb{R}/2r\mathbb{Z}$).  
  In particular, the bijection $\mathbb{R}/r\mathbb{Z}\to\mathbb{R}/\mathbb{Z}$ (resp. $\mathbb{R}/2r\mathbb{Z}\to\mathbb{R}/\mathbb{Z}$) given by $x\mapsto x/r$ (resp. $x\mapsto x/(2r)$) induces isomorphisms $D_r:\abs{\operatorname{sd}_rX[-]}\to\abs{X[-]}$ (resp. $D_r^{\operatorname{e}}:\abs{\operatorname{sd}_r^{\operatorname{e}}X[-]}\to\abs{X[-]}$) and $d_r:G_r\to G=\operatorname{Homeo}\mathbb{R}/\mathbb{Z}$ (resp. $d_r^{\operatorname{e}}:G_r^{\operatorname{e}}\to G=\operatorname{Homeo}\mathbb{R}/\mathbb{Z}$) such that 
  the diagram
  \begin{displaymath}
  \begin{CD}
  G_r\times\abs{\operatorname{sd}_rX[-]} @>>> \abs{\operatorname{sd}_rX[-]}  \\ 
  @VV{D_r\times d_r}V @VV{D_r}V  \\
  G\times\abs{X[-]} @>>> \abs{X[-]} 
  \end{CD}
  \end{displaymath}  
  \begin{displaymath}
  \begin{CD}
  (\text{resp. } G_r\times\abs{\operatorname{sd}_r^{\operatorname{e}}X[-]} @>>> \abs{\operatorname{sd}_r^{\operatorname{e}}X[-]}  \\ 
  @VV{D_r^{\operatorname{e}}\times d_r^{\operatorname{e}}}V @VV{D_r^{\operatorname{e}}}V  \\
  G\times\abs{X[-]} @>>> \abs{X[-]} )
  \end{CD}
  \end{displaymath}
  commutes. 
  \end{thm}
  
  Finally, we explain how $\operatorname{sd}_r^{\operatorname{e}}X[-]$ admits simplicial actions by $D_r$ and $C_r$, and hence defines simplicial sets $(\operatorname{sd}_r^{\operatorname{e}}X[-])^{D_r}$ and $(\operatorname{sd}_r^{\operatorname{e}}X[-])^{C_r}$, respectively, and see that the latter one again has the structure of a dihedral set. 

  \section{Drinfeld's method}
  In this section we review Drinfeld's results (\cite{d}) on the realization of simplicial sets and cyclic sets
  .  
  
  \subsection{Simplicial sets}
  Let $X[-]:\Delta^{\operatorname{op}}\to\operatorname{Sets}$ be a simplicial set. 
  We extend $X[-]$ to $\widetilde{X}[-]:\Delta^{\operatorname{op}}_{\operatorname{big}}\to\operatorname{Sets}$ as follows. 
  For every object $\mathcal{A}=\{a_0<\cdots<a_n\}$ of $\Delta_{\operatorname{big}}$, there is a unique isomorphism $i_{\mathcal{A}}\to[n]$ that sends $a_i$ to $i$. 
  We define $\widetilde{X}[A]=X[\operatorname{card}A-1]$ and $\widetilde{X}[f]=X[\widetilde{f}]$, where $f:\mathcal{A}\to\mathcal{B}$ is a map in $\Delta_{\operatorname{big}}$ and $\widetilde{f}$ is the unique map in $\Delta$ that makes the following diagram commute: 
  \begin{displaymath}
  \begin{CD}
  \mathcal{A} @>{i_{\mathcal{A}}}>> [m]   \\ 
  @V{f}VV @V{\widetilde{f}}VV  \\
  \mathcal{B} @>{i_{\mathcal{B}}}>> [n]
  \end{CD}
  \end{displaymath} 
  We note that if $\widetilde{X}^\prime[-]$ is another extension, i.e. a functor $\Delta_{\operatorname{big}}^{\operatorname{op}}\to\operatorname{Sets}$ that is identical to $X[-]$ on $\Delta^{\operatorname{op}}$, then there exists a unique natural isomorphism $\kappa:\widetilde{X}[-]\to\widetilde{X}^\prime[-]$ such that $\kappa\mid_{\Delta^{\operatorname{op}}}=\operatorname{id}_{X[-]}.$  
  For instance, the extension $\widetilde{\Delta[n]}[-]$ of the standard simplicial $n$-simplex is identified with the functor $\mathcal{A}\mapsto\operatorname{Hom}_{\Delta_{\operatorname{big}}}(\mathcal{A},[n])$. 
  
  We topologize the filtered colimit $\operatorname{colim}_{F\in\mathcal{F}}\widetilde{X}[\pi_0([0,1]\setminus F)]$ by the metric $d$ defined as follows. 
  For every $F\in\mathcal{F}$, we write $\mu_F$ for the measure on $\pi_0([0,1]\setminus F)$ defined by $\mu_F(A)=\sum_{c\in A}(\text{length of c})$. 
  Take two elements of $\operatorname{colim}_{F\in\mathcal{F}}\widetilde{X}[\pi_0([0,1]\setminus F)]$. 
  Since $\mathcal{F}$ is a filtering, we may assume these two be represented by elements $u$ and $v$ of $\widetilde{X}[\pi_0([0,1]\setminus F)]$ with some common $F\in\mathcal{F}$. 
  (We have to check that the following definition is independent of the choice of such an $F$. 
  See below.) 
  We define the distance of the two elements to be the minimum of $\mu_F(\pi_0([0,1]\setminus F)\setminus A)$ with respect to subsets $A$ of $\pi_0([0,1]\setminus F)$ such that 
  \begin{description}
  \item[$(\ast)$] the map $\widetilde{X}[\pi_0([0,1]\setminus F)]\to\widetilde{X}[A]$ takes $u$ and $v$ to an identical element. 
  \end{description}
  (If there does not exist such an $A$, we set the distance to be $1$.) \\ ~ \\
  {\bf Well-definedness of $d$}. 
  Suppose $F\subset F^\prime\subset[0,1]$ and let $u^\prime$ and $v^\prime\in\widetilde{X}[\pi_0([0,1]\setminus F^\prime)]$ be the images of $u$ and $v\in\widetilde{X}[\pi_0([0,1]\setminus F)]$. 
  We write $d_F$ for the distance defined by using $F$ and $d_{F^\prime}$ for the one by $F^\prime$. 
  We have to show that $d_F(u,v)=d_{F^\prime}(u^\prime,v^\prime)$. 
  For an $A^\prime\subset\pi_0([0,1]\setminus F^\prime)$ satisfying $(\ast)$, we set $A=\{[x]\mid x\in[0,1], [x]^\prime\in A^\prime\}\subset\pi_0([0,1]\setminus F)$, where $[-]$ and $[-]^\prime$ denotes the class in $\pi_0([0,1]\setminus F)$ and $\pi_0([0,1]\setminus F^\prime)$, respectively. 
  Then there is a canonical map $a^\prime:A^\prime\to A$, $[x]^\prime\mapsto[x]$, and we have $\mu_F(\pi_0([0,1]\setminus F)\setminus A)\leq\mu_F(\pi_0([0,1]\setminus F^\prime)\setminus A^\prime)$. 
  We also choose a map $a:A\to A^\prime$ that takes $[x]$ to $[y]^\prime$ if $[x]=[y]$ and $[y]^\prime\in A^\prime$. 
  Then, since $a^\prime\circ a=\operatorname{id}_A$, the map $\widetilde{X}[a^\prime]$ is injective. 
  As the diagram 
  \begin{displaymath}
  \begin{CD}
  \widetilde{X}[\pi_0([0,1]\setminus F)] @>>> \widetilde{X}[\pi_0([0,1]\setminus F^\prime)]   \\ 
  @VVV @VVV  \\
  \widetilde{X}[A] @>>{\widetilde{X}[a^\prime]}> \widetilde{X}[A^\prime]
  \end{CD}
  \end{displaymath}    
  commutes, the images of $u$, $v\in\widetilde{X}[\pi_0([0,1]\setminus F)]$ in $\widetilde{X}[A]$ are sent by $\widetilde{X}[a^\prime]$ to the images of $u^\prime$ and $v^\prime$ in $\widetilde{X}[A^\prime]$, which are assumed to be identical. 
  The injectivity deduces that $A$ satisfies $(\ast)$, and thus we have $d_F(u,v)\leq\mu_F(\pi_0([0,1]\setminus F)\setminus A)\leq\mu_F(\pi_0([0,1]\setminus F^\prime)\setminus A^\prime)$, for every $A^\prime$ satisfying $(\ast)$. 
  Hence $d_F(u,v)\leq d_{F^\prime}(u^\prime,v^\prime)$. 
  Conversely, if $A\subset\pi_0([0,1]\setminus F)$ satisfies $(\ast)$, then $A^\prime=\{[x]^\prime\mid[x]\in A\}\subset\pi_0([0,1]\setminus F^\prime)$ is a subset such that $(\ast)$ holds and $\mu_{F^\prime}(\pi_0([0,1]\setminus F^\prime)\setminus A^\prime)=\mu_F(\pi_0([0,1]\setminus F)\setminus A)$. 
  This implies the inequality in the opposite direction, and we obtain $d_F(u,v)=d_{F^\prime}(u^\prime,v^\prime)$.

  Let us also verify that $d$ is really a metric. 
  By definition, $d$ is symmetric and takes non-negative values. 
  If $d(u,v)=\mu_F(\pi_0([0,1]\setminus F)\setminus A)=0$ then $A$ must be the whole $\pi_0([0,1]\setminus F)$. 
  This means that $u$ and $v$ are identical in $\widetilde{X}[\pi_0([0,1]\setminus F)]=\widetilde{X}[A]$. 
  We have left to check the triangle inequality. 
  Take $u$, $v$, and $w\in\widetilde{X}[\pi_0([0,1]\setminus F)]$, and suppose $d(u,v)=\mu_F(\pi_0([0,1]\setminus F)\setminus A)$ and $d(v,w)=\mu_F(\pi_0([0,1]\setminus F)\setminus B)$. 
  (If there does not exist $A$ or $B$ satisfying $(\ast)$, then $d(u,v)+d(v,w)$ is equal to or larger than $1$, whereas $d(u,w)$ is always equal to or smaller than $1$. 
  Hence we are done.) 
  Note, in general, that $\mu_F(\pi_0([0,1]\setminus F)\setminus A)+\mu_F(\pi_0([0,1]\setminus F)\setminus B)\geq\mu_F(\pi_0([0,1]\setminus F)\setminus(A\cap B))$. 
  If $A\cap B=\emptyset$, then $\mu_F(\pi_0([0,1]\setminus F)\setminus A)+\mu_F(\pi_0([0,1]\setminus F)\setminus B)\geq\mu_F(\pi_0([0,1]\setminus F)\setminus(A\cap B))=1\geq d(u,w)$ and we are done. 
  Suppose $A\cap B\neq\emptyset$. 
  The map $\widetilde{X}[\pi_0([0,1]\setminus F)]\to\widetilde{X}[A]$ (resp. $\widetilde{X}[\pi_0([0,1]\setminus F)]\to\widetilde{X}[B]$) takes $u$ and $v$ (resp. $v$ and $w$) to an identical element in $\widetilde{X}[A]$ (resp. $\widetilde{X}[B]$). 
  Since the diagram 
  \begin{displaymath}
  \begin{CD}
  A\cap B @>>> A   \\ 
  @VVV @VVV  \\
  B @>>> \pi_0([0,1]\setminus F)
  \end{CD}
  \end{displaymath}  
  commutes, the map $\widetilde{X}[\pi_0([0,1]\setminus F)]\to\widetilde{X}[A\cap B]$ tales $u$ and $v$ (resp. $v$ and $w$) to an identical element in $\widetilde{X}[A\cap B]$. 
  Hence the images of $u$ and $w$ coincide in $\widetilde{X}[A\cap B]$, so that $d(u,w)\leq\mu_F(\pi_0([0,1]\setminus F)\setminus(A\cap B))\leq\mu_F(\pi_0([0,1]\setminus F)\setminus A)+\mu_F(\pi_0([0,1]\setminus F)\setminus B)=d(u,v)+d(v,w)$. 
  \begin{flushright}
  $\Box$
  \end{flushright}
  
  \begin{thm}[Drinfeld \cite{d}]
  \label{newformulation2} 
  There is a canonical homeomorphism $$\abs{X[-]}\stackrel{\cong}{\to}\operatorname{colim}_{F\in\mathcal{F}}\widetilde{X}[\pi_0([0,1]\setminus F)]. $$
  \end{thm}
  \begin{proof}
  We note that there is a homeomorphism between $\Delta[n]$ and $\operatorname{Sim}^n=\{(x_1,\ldots,x_n)\in[0,1]^n\mid x_1\leq\cdots\leq x_n\}$ given by $\Delta[n]\to\operatorname{Sim}^n$, $(z_0,\ldots,z_n)\mapsto(z_0,z_0+z_1,\ldots,z_0+\cdots+z_{n-1})$, and $\operatorname{Sim}^n\to\Delta[n]$, $(x_1,\ldots,x_n)\mapsto(x_1,x_2-x_1,$ $\ldots,x_n-x_{n-1},1-x_n)$.   
  In turn, the set $\operatorname{Sim}^n$ can be written as the filtered colimit $\operatorname{colim}_{F\in\mathcal{F}}\widetilde{\Delta[n]}[\pi_0([0,1]\setminus F)]$ by identifying $\mathbf{x}=(x_1,\ldots,x_n)\in\operatorname{Sim}^n$ with the piecewise constant function $f^{\mathbf{x}}:[0,1]\to[n]$ defined by $f^{\mathbf{x}}(x)=i$ for $x_i<x<x_{i+1}$ (we set $x_0=0$ and $x_{n+1}=1$).  
  Since $X[-]\cong\operatorname{colim}_{\Delta[n][-]\to X[-]}\Delta[n][-]$, we have    
  \begin{align*}
  \abs{X[-]}&=\operatorname{colim}_{\Delta[n][-]\to X[-]}\Delta[n]\\
  &\stackrel{\cong}{\to}\operatorname{colim}_{\Delta[n][-]\to X[-]}(\operatorname{colim}_{F\in\mathcal{F}}\widetilde{\Delta[n]}[\pi_0([0,1]\setminus F)])\\
  &\stackrel{\cong}{\to}\operatorname{colim}_{F\in\mathcal{F}}(\operatorname{colim}_{\Delta[n][-]\to X[-]}\widetilde{\Delta[n]}[\pi_0([0,1]\setminus F)])\\
  &\stackrel{\cong}{\to}\operatorname{colim}_{F\in\mathcal{F}}\widetilde{X}[\pi_0([0,1]\setminus F)].    
  \end{align*}
  (Here $\stackrel{\cong}{\to}$ stands for the canonical set bijections.) 
  Thus we have obtained a continuous bijection $\abs{X[-]}\stackrel{\cong}{\to}\operatorname{colim}_{F\in\mathcal{F}}\widetilde{X}[\pi_0([0,1]\setminus F)]$ (the continuity is proved as Lemma \ref{lemma2.2} below). 
  The target space is Hausdorff as its topology comes from a metric.  
  Moreover, if $X[-]$ is a finite simplicial set (i.e. has only finitely many non-degenerate simplices), then $\abs{X[-]}$ is compact (it is a quotient of 
  the finite disjoint union of simplices $\Delta[n]$ with one $\Delta[n]$ for each non-degenerate $n$-simplex of $X[-]$), so that this bijection should be a homeomorphism in this case. 
  This implies the statement for general $X[-]$, since $X[-]$ is the colimit of its finite simplicial subsets. 
  \end{proof}
  \begin{lemma}
  \label{lemma2.2}
  The bijection $\abs{X[-]}\to\operatorname{colim}_{F\in\mathcal{F}}\widetilde{X}[\pi_0([0,1]\setminus F)]$ is continuous. 
  \end{lemma}
  \begin{proof}
  It saffices to show that for every index $\Delta[n][-]\to X[-]$, the map $\operatorname{Sim}^n\to\operatorname{colim}_{F\in\mathcal{F}}(\operatorname{colim}_{\Delta[m][-]\to X[-]}\widetilde{\Delta[m]}[\pi_0([0,1]\setminus F)]$, $\mathbf{x}=(x_1,\ldots,x_n)\mapsto(\text{class of }f^{\mathbf{x}})\in\widetilde{\Delta[n]}[\pi_0([0,1]\setminus\{x_1,\ldots,x_n\})]$ satisfies the following: 
  
  \textit{For any $\varepsilon>0$, there exists some $\delta>0$ such that for any $\mathbf{x}=(x_1,\ldots,x_n),\mathbf{y}=(y_1,\ldots,y_n)\in\operatorname{Sim}^n$, if $d^\prime(\mathbf{x},\mathbf{y})<\delta$ then $d(f^{\mathbf{x}},f^{\mathbf{y}})<\varepsilon$, where $d^\prime(\mathbf{x},\mathbf{y})
  $ is the standard euclidean metric.}   
   

  Let $F=\{x_1,\ldots,x_n,y_1,\ldots,y_n\}$ and denote by $A$ the subset of $\pi_0([0,1]\setminus F)$ formed by the classes of those points $x\in[0,1]$ such that $f^{\mathbf{x}}(x)$ and $f^{\mathbf{y}}(x)$ coincide. 
  Then the images of $f^{\mathbf{x}}$ and $f^{\mathbf{y}}$ by the map $\widetilde{\Delta[n]}[\pi_0([0,1]\setminus F)]\to\widetilde{\Delta[n]}[A]$ are identical, so that $d(f^{\mathbf{x}},f^{\mathbf{y}})\leq\mu_F(\pi_0([0,1]\setminus F)\setminus A)$. 
  By construction, $f^{\mathbf{x}}(x)\neq f^{\mathbf{y}}(x)$ happens only if $x$ is in $(x_i,y_i)$ or $(y_i,x_i)$ for some $i$. 
  This implies that $\mu_F(\pi_0([0,1]\setminus F)\setminus A)\leq\sum^n_{i=1}\abs{x_i-y_i}\leq nd^\prime(\mathbf{x},\mathbf{y})$, and tells that we may wish to take $\delta=\varepsilon/n$. 
  \end{proof}
  {\bf Proof of Corollary \ref{cor1.2}}. \\
  \underline{Proof of claim 1}. We prove this in several steps.
  \begin{description}
  \item[Step 1.] \textit{The canonical map $$\abs{\Delta[m][-]\times\Delta[n][-]}\to\abs{\Delta[m][-]}\times\abs{\Delta[n][-]}$$ is a homeomorphism}.  
  \end{description}
  By the theorem, this map is a continuous bijection, with the target space Hausdorff. 
  Moreover, the domain space is compact since $\Delta[m][-]\times\Delta[n][-]$ is a finite simplicial set. 
  Thus the claim follows. 
  \begin{description}
  \item[Step 2.] \textit{Geometric realization preserves finite products.} 
  \end{description}
  We use Step 1 and the properties of the categories $\operatorname{Sets}^{\Delta^{\operatorname{op}}}$ of simplicial sets and $\mathcal{K}$ of $k$-spaces that product commutes with colimits. 
  Notice also that the singular set functor $\mathcal{K}\to\operatorname{Sets}^{\Delta^{\operatorname{op}}}$ that assigns to a $k$-space its singular simplicial set is a right adjoint functor to the geometric realization functor, and hence geometric realization commutes with colimits. 
  Let $X[-]$ and $Y[-]$ be simplicial sets. 
  Then we have  
  \begin{align*}
  \abs{X[-]\times Y[-]}&\stackrel{\cong}{\leftarrow}\abs{(\operatorname{colim}_{\Delta[m][-]\to X[-]}\Delta[m][-])\times(\operatorname{colim}_{\Delta[n][-]\to Y[-]}\Delta[n][-])}\\
  &\stackrel{\cong}{\leftarrow}\abs{\operatorname{colim}_{\Delta[m][-]\to X[-]}(\operatorname{colim}_{\Delta[n][-]\to Y[-]}(\Delta[m][-]\times\Delta[n][-]))}\\
  &\stackrel{\cong}{\leftarrow}\operatorname{colim}_{\Delta[m][-]\to X[-]}(\operatorname{colim}_{\Delta[n][-]\to Y[-]}(\abs{\Delta[m][-]\times\Delta[n][-]}))\\
  &\stackrel{\cong}{\to}\operatorname{colim}_{\Delta[m][-]\to X[-]}(\operatorname{colim}_{\Delta[n][-]\to Y[-]}(\abs{\Delta[m][-]}\times\abs{\Delta[n][-]}))\\
  &\stackrel{\cong}{\to}(\operatorname{colim}_{\Delta[m][-]\to X[-]}\Delta[m])\times(\operatorname{colim}_{\Delta[n][-]\to Y[-]}\Delta[n])\\
  &=\abs{X[-]}\times\abs{Y[-]},  
  \end{align*}  
  with $\stackrel{\cong}{\to}$ standing for the canonical $k$-space homeomorphisms. 
  \begin{description}
  \item[Step 3.] \textit{Geometric realization preserves finite limits.} 
  \end{description}
  As every finite limit can be written as an equalizer of finite products, it suffices, by Step 2, to show that geometric realization preserves an equalizer diagram of simplicial sets $X[-]\to Y[-]\rightrightarrows Z[-]$. 
  By the theorem, $\abs{X[-]}\to\abs{Y[-]}\rightrightarrows\abs{Z[-]}$ is an equalizer in $\operatorname{Sets}$. 
  Moreover, the topology of $\abs{X[-]}$ coincides with the subspace topology in $\abs{Y[-]}$, since the metric on $\abs{X[-]}$ is identical to the restriction of the metric on $\abs{Y[-]}$ to the subspace $\abs{X[-]}\subset\abs{Y[-]}$. 
  This completes the proof.\\ ~ \\ 
  \underline{Proof of claim 2}. An orientation preserving homeomorphism $\alpha:[0,1]\to[0,1]$ gives rise to an isomorphism of linearly ordered sets $\alpha_F:\pi_0([0,1]\setminus F)\to\pi_0([0,1]\setminus \alpha(F))$ for every $F\in\mathcal{F}$. 
  The action by $\alpha$  
  $$\rho_\alpha:\abs{X[-]}=\operatorname{colim}_{F\in\mathcal{F}}\widetilde{X}[\pi_0([0,1]\setminus F)]\to\operatorname{colim}_{F\in\mathcal{F}}\widetilde{X}[\pi_0([0,1]\setminus F)]$$
  is given by $\rho_\alpha\circ\operatorname{in}_F=\operatorname{in}_{\alpha(F)}\circ\widetilde{X}[\alpha_F^{-1}].$ 
  If we express $X[-]$ as $\operatorname{colim}_{\Delta[n][-]\to X[-]}\Delta[n]$, the isomorphism $\widetilde{X}[\alpha_F^{-1}]:\widetilde{X}[\pi_0([0,1]\setminus F)]= \operatorname{colim}_{\Delta[n][-]\to X[-]}\widetilde{\Delta[n]}[\pi_0([0,1]\setminus F)]\to\operatorname{colim}_{\Delta[n][-]\to X[-]}\widetilde{\Delta[n]}[\pi_0([0,1]\setminus\alpha(F))]=\widetilde{X}[\pi_0([0,1]\setminus\alpha(F))]$ is given by the maps $\widetilde{\Delta[n]}[\pi_0([0,1]\setminus F)]\to\widetilde{\Delta[n]}[\pi_0([0,1]\setminus\alpha(F))]$, one for each index $\Delta[n][-]$ $\to X[-]$, that take a map $f:\pi_0([0,1]\setminus F)\to[n]$ to the map $f\circ\alpha_F^{-1}:\pi_0([0,1]\setminus\alpha(F))\to[n]$. 
  Note that if $F=\{x_1,\ldots,x_n\}$ and $f=f^{\mathbf{x}}$, where $\mathbf{x}=(x_1,\ldots,x_n)\in\operatorname{Sim}^n$, then $f\circ\alpha_F^{-1}=f^{\alpha\times\cdots\times\alpha(\mathbf{x})}$. 
  Thus, in the expression $\abs{X[-]}=\operatorname{colim}_{\Delta[n][-]\to X[-]}$ $\operatorname{Sim}^n$, the action $\rho_\alpha$ is obtained by 
  taking $\mathbf{x}=(x_1,\ldots,x_n)\in\operatorname{Sim}^n$ to $\alpha\times\cdots\times\alpha(\mathbf{x})\in\operatorname{Sim}^n$ for every $\Delta[n][-]\to X[-]$. 
  
  We wish to show that the map $\mu:\operatorname{Homeo}([0,1],\partial[0,1])\times\operatorname{colim}_{\Delta[n][-]\to X[-]}\operatorname{Sim}^n$ $\to\operatorname{colim}_{\Delta[n][-]\to X[-]}\operatorname{Sim}^n$ is continuous. 
  Since, by adjunction, product commutes with colimits in $\mathcal{K}$, it suffices to show that the map $\operatorname{Homeo}([0,1],\partial[0,1])$ $\times\operatorname{Sim}^n\to\operatorname{Sim}^n$, $(\alpha,\mathbf{x})$ $\mapsto\alpha\times\cdots\times\alpha(\mathbf{x})$, is continuous for every index $\Delta[n][-]\to X[-]$. 
  Again by adjunction, thinking of this map is equivalent to considering the map $\operatorname{Homeo}([0,1],$ $\partial[0,1])\to\operatorname{Hom}_{\mathcal{K}}(\operatorname{Sim}^n,\operatorname{Sim}^n)$, $\alpha\mapsto\alpha\times\cdots\times\alpha\mid_{\operatorname{Sim^n}}$, whose continuity is proved 
  as the following lemma. 
  \begin{flushright}
  $\Box$
  \end{flushright}
  \begin{lemma}
  \label{lemma2.3}
  The map $\mu_n:\operatorname{Homeo}([0,1],\partial[0,1])\to\operatorname{Hom}_{\mathcal{K}}(\operatorname{Sim}^n,\operatorname{Sim}^n)$, $\alpha\mapsto\alpha\times\cdots\times\alpha\mid_{\operatorname{Sim}^n}$, is continuous with respect to the standard $k$-space topologies on both sides. 
  \end{lemma}
  \begin{proof}
  Remember that the subbasis of the target space is given by the subsets $N(h,U)$ $=\{f:\operatorname{Sim}^n\to\operatorname{Sim}^n\mid f(h(K))\subset U\}$ where $h:K\to\operatorname{Sim}^n$ is a continuous map from a compact Hausdorff space $K$ and where $U$ is an open set of $\operatorname{Sim}^n$. 
  Hence it suffices to show that $\mu_n^{-1}N(h,U)$ is open in $\operatorname{Homeo}([0,1],\partial[0,1])$. 
  To this end, we fix an arbitrary $\alpha\in\mu_n^{-1}N(h,U)$ and will show that there is an open neighbourhood $N(\alpha)$ of $\alpha$ in $\operatorname{Homeo}([0,1],\partial[0,1])$ such taht $N(\alpha)\subset\mu_n^{-1}N(h,U)$. 
  Since $U\subset\operatorname{Sim}^n$ is open, for every $x\in h(K)$, there exists a positive real number $\varepsilon_x$ such that $B_x^{\prime\prime}=\{y\in\operatorname{Sim}^n\mid\abs{y-\mu_n(\alpha)(x)}<\varepsilon_x\}$ is contained in $U$. 
  We take a smaller ball $B_x^\prime=\{y\in\operatorname{Sim}^n\mid\abs{y-\mu_n(\alpha)(x)}<\varepsilon_x/n\}$ in $B_x^{\prime\prime}$, and put $B_x=\mu_n(\alpha)^{-1}(B_x^\prime)\cap h(K)$. 
  Then $\{B_x\}_{x\in h(K)}$ forms an open cover for $h(K)$. 
  A compactness argument tells us that we can choose finite $x^{(1)},\ldots,x^{(l)}\in h(K)$ such that $h(K)=\bigcup_{j+1}^{l}B_{x^{(j)}}$. 
  If $\overline{B_x}$ denotes the closure of $B_x$ in $h(K)$, we also have $h(K)=\bigcup_{j+1}^{l}\overline{B_{x^{(j)}}}$. 
  Note that $\overline{B_x}$ is compact (because it is a closed set in a compact set). 
  We let $\iota_x:\overline{B_x}\to h(K)\to\operatorname{Sim}^n$ be the inclusion, and consider for every $i=1,\ldots,n$ and $j=1,\ldots,l$, the set $N^\prime(p_i\circ\iota_{x^{(j)}},p_i(B^\prime_{x^{(j)}}))=\{\beta\in\operatorname{Homeo}([0,1],\partial[0,1])\mid\beta(p_i(\overline{B_{x^{(j)}}}))\subset p_i(B^\prime_{x^{(j)}})\}$, where $p_i:\operatorname{Sim}^n\to[0,1]$ is the projection onto the $i$-th component. 
  Then $N^\prime(p_i\circ\iota_{x^{(j)}},p_i(B^\prime_{x^{(j)}}))$ is an open set in $\operatorname{Homeo}([0,1],\partial[0,1])$ containing $\alpha$. 
  We also have $\bigcap_{1\leq i\leq n, 1\leq j\leq l}N^\prime(p_i\circ\iota_{x^{(j)}},p_i(B^\prime_{x^{(j)}}))\subset\mu_n^{-1}(N(h,U))$. 
  Indeed, let $\beta$ be in the left-hand side and take $x\in h(K)$ arbitrarily. 
  Then there is some $j$ such that $x=(x_1,\ldots,x_n)\in\overline{B_{x^{(j)}}}$. 
  For every $1\leq i\leq n$, we have $\beta(p_i(x))=\beta(x_i)\in p_i(B^\prime_{x^{(j)}})\subset\{y_i\in[0,1]\mid\abs{y_i-p_i(\mu_n(\alpha)(x^{(j)}))}<\varepsilon_{x^{(j)}}/n\}$. 
  Hence, $\abs{\mu_n(\beta)(x)-\mu_n(\alpha)(x^{(j)})}^2=\sum_{i+1}^{n}(\beta(x_i)-\alpha(x_i^{(j)}))^2\leq\sum_{i=1}^n\varepsilon^2_{x^{(j)}}/n^2=\varepsilon_{x^{(j)}}^2/n<\varepsilon_{x^{(j)}}^2$. 
  Therefore we see $\mu_n(\beta)(x)\in B^{\prime\prime}_{x^{(j)}}\subset U$. 
  This implies $\beta\in\mu_n^{-1}(N(h,U))$. 
  Thus we take $N(\alpha)=\bigcap_{1\leq i\leq n, 1\leq j\leq l}N^\prime(p_i\circ\iota_{x^{(j)}},p_i(B^\prime_{x^{(j)}}))$, obtaining the desired conclusion. 
  \end{proof} 

  \subsection{Cyclic sets}
  Let $X[-]:(\Delta C)^{\operatorname{op}}\to\operatorname{Sets}$ be a cyclic set. 
  We choose one isomorphism $i_\lambda:\lambda\to[m_\lambda]_{\text{cyc}}$ for each $\lambda\in\operatorname{ob}\Delta_{\text{big}}C$. 
  (We let $i_\lambda=\operatorname{id}_{\lambda}$ if $\lambda\in\operatorname{ob}\Delta C$.) 
  We extend the cyclic set $X[-]
  $ to a functor $\widetilde{X}[-]:(\Delta_{\text{big}}C)^{\operatorname{op}}\to\operatorname{Sets}$, by defining on objects $\widetilde{X}[\lambda]=X[m_\lambda]_{\text{cyc}}$ and on morphisms $\widetilde{X}[f]=X[\widetilde{f}]:X[m_\mu]_{\text{cyc}}\to X[m_\lambda]_{\text{cyc}}$, where $f:\lambda\to\mu$ is a map in $\Delta_{\text{big}}C$ and $\widetilde{f}$ is the unique map in $\Delta C$ that makes the following diagram commute: 
  \begin{displaymath}
  \begin{CD}
  \lambda @>{i_{\lambda}}>> [m_\lambda]_{\text{cyc}}   \\ 
  @V{f}VV @V{\widetilde{f}}VV  \\
  \mu @>{i_{\mu}}>> [m_\mu]_{\text{cyc}}
  \end{CD}
  \end{displaymath}
  For example, if $X[-]$ is the standard cyclic set $\Lambda[n][-]=\operatorname{Hom}_{\Delta C}([-],[n]_{\operatorname{cyc}})$, the extension $\widetilde{\Lambda[n]}[-]$ is given by $\lambda\mapsto\operatorname{Hom}_{\Delta_{\operatorname{big}}C}(\lambda,[n]_{\operatorname{cyc}})$. 
    
  If $\widetilde{X}^\prime[-]$ is another extension, i.e. a functor $(\Delta_{\text{big}}C)^{\operatorname{op}}\to\operatorname{Sets}$ that is identical to $X[-]$ on $(\Delta C)^{\operatorname{op}}$, then there exists a unique natural isomorphism $\kappa:\widetilde{X}[-]\to\widetilde{X}^\prime[-]$ such that $\kappa\mid_{(\Delta C)^{\operatorname{op}}}=\operatorname{id}_{X[-]}.$  
  Indeed, if $\kappa$ is such a natural isomorphism then there is a commutative diagram 
  \begin{displaymath}
  \begin{CD}
  \widetilde{X}(\lambda) @>{\kappa(\lambda)}>> \widetilde{X}^\prime(\lambda)   \\ 
  @A{\widetilde{X}(i_{\lambda})}AA @A{\widetilde{X}^\prime(i_{\lambda})}AA  \\
  X[m_\lambda]_{\text{cyc}} @= X[m_\lambda]_{\text{cyc}}  
  \end{CD}
  \end{displaymath}  
  for each $\lambda\in\operatorname{ob}(\Delta_{\text{big}}C)^{\operatorname{op}}$. 
  This forces $\kappa(\lambda)=\widetilde{X}^\prime[i_{\lambda}]\circ[\widetilde{X}(i_{\lambda})]^{-1}$. 
  \\ ~ \\
  %
  %
  {\bf Proof of Theorem \ref{thm1.3}}. 
  Remember that the realization of $X[-]$ is the realization of the underlying simplicial set $X\mid_{\Delta^{\operatorname{op}}}[-]$. 
  This means $$\abs{X[-]}=\operatorname{colim}_{F\in\mathcal{F}}\widetilde{(X\mid_{\Delta^{\operatorname{op}}})}[\pi_0([0,1]\setminus F)]. $$ 
  Here $\widetilde{(X\mid_{\Delta^{\operatorname{op}}})}[-]$ is the extension of $X\mid_{\Delta^{\operatorname{op}}}[-]$ to $\Delta^{\operatorname{op}}_{\operatorname{big}}$, which equals the restriction $\widetilde{X}\mid_{\Delta^{\operatorname{op}}_{\operatorname{big}}}[-]$ of $\widetilde{X}[-]$ to $\Delta^{\operatorname{op}}_{\operatorname{big}}$ along with the functor $\Delta^{\operatorname{op}}_{\operatorname{big}}\to(\Delta_{\text{big}}C)^{\operatorname{op}}$, $\mathcal{A}\mapsto\mathcal{A}_{\operatorname{cyc}}$. 
  Since it makes no change on the colimit to take into account only those $F$ containing $0$ and $1$, we have  
  \begin{align*}
  \abs{X[-]}&=\operatorname{colim}_{F\in\mathcal{F}}\widetilde{X}[\pi_0([0,1]\setminus F)_{\operatorname{cyc}}]\\ 
  &\stackrel{\cong}{\leftarrow}\operatorname{colim}_{0,1\in F\in\mathcal{F}}\widetilde{X}[\pi_0([0,1]\setminus F)_{\operatorname{cyc}}]\\
  &\stackrel{\cong}{\to}\operatorname{colim}_{0,1\in F\in\mathcal{F}}\widetilde{X}[\pi_0(\mathbb{R}/\mathbb{Z}\setminus\overline{F})]\\
  &\stackrel{\cong}{\to}\operatorname{colim}_{0\in F\in\mathcal{F}^\prime}\widetilde{X}[\pi_0(\mathbb{R}/\mathbb{Z}\setminus F)]\\
  &\stackrel{\cong}{\to}\operatorname{colim}_{F\in\mathcal{F}^\prime}\widetilde{X}[\pi_0(\mathbb{R}/\mathbb{Z}\setminus F)],  
  \end{align*}
  where $\stackrel{\cong}{\to}$ are the canonical set bijections. 
  We remark that the distance of $u$ and $v\in\widetilde{X}[\pi_0(\mathbb{R}/\mathbb{Z}\setminus F)]$, coming from the metric on $\abs{X[-]}=\abs{X\mid_{\Delta^{\operatorname{op}}}[-]}$ via the bijections above, is given by the minimum of $\mu^{\prime}_F(\pi_0(\mathbb{R}/\mathbb{Z}\setminus F)\setminus A)$, where $\mu^{\prime}_F$ is a similar measure on $\pi_0(\mathbb{R}/\mathbb{Z}\setminus F)$, and where $A$ runs through subsets of $\pi_0(\mathbb{R}/\mathbb{Z}\setminus F)$ such that the map $\widetilde{X}[\pi_0(\mathbb{R}/\mathbb{Z}\setminus F)]\to\widetilde{X}[A]$ takes $u$ and $v$ to an identical element. 

  The construction of the action by $\operatorname{Homeo}^+\mathbb{R}/\mathbb{Z}$ is analogous to Corollary \ref{cor1.2}-2. 
  To prove its continuity, we need a lemma. 
  \begin{lemma} 
  \label{lemma2.4}
  The geometric realization of the standard cyclic set $\Lambda[n][-]$ is given by the space $\operatorname{Sim}^n_{\operatorname{cyc}}$ of points $(x_0,\ldots,x_n)$ of $(\mathbb{R}/\mathbb{Z})^{n+1}$ such that $x_0,\ldots,x_n$ are in the correct cyclic order. 
  \end{lemma}
  \begin{proof}
  We have a bijection $\operatorname{Sim}^n_{\operatorname{cyc}}\to\operatorname{colim}_{F\in\mathcal{F}^{\prime}}
  \widetilde{\Lambda[n]}[\pi_0(\mathbb{R}/\mathbb{Z}\setminus F)]
  =\abs{\Lambda[n][-]}$ under which a point $\mathbf{x}=(x_0,\ldots,x_n)\in\operatorname{Sim}^n_{\operatorname{cyc}}$ corresponds to the piecewise constant function $f^{\mathbf{x}}
  :\mathbb{R}/\mathbb{Z}\to[n]_{\operatorname{cyc}}$ that takes the constant value $i/(n+1)$ on the arc from $x_i$ to $x_{i+1}\in\mathbb{R}/\mathbb{Z}$. 
  The metric on $\abs{\Lambda[n][-]}$ described above makes it possible to apply an argument analogous to the proof of Lemma \ref{lemma2.2} to proving that this bijection is continuous. 
  Since the domain space is compact and the target Hausdorff, the claim follows.    
  \end{proof}
  If we express the cyclic set $X[-]$ as $X[-]=\operatorname{colim}_{\Lambda[n][-]\to X[-]}\Lambda[n][-]$, the realization is given by $\abs{X[-]}=\operatorname{colim}_{\Lambda[n][-]\to X[-]}\abs{\Lambda[n][-]}=\operatorname{colim}_{\Lambda[n][-]\to X[-]}\operatorname{Sim}^n_{\operatorname{cyc}}$, since geometric realization commutes with colimits by adjunction. 
  Then the action by $\alpha\in\operatorname{Homeo}^+\mathbb{R}/\mathbb{Z}$ is given by taking $\mathbf{x}\in\operatorname{Sim}^n_{\operatorname{cyc}}$ to $\alpha\times\cdots\times\alpha(\mathbf{x})\in\operatorname{Sim}^n_{\operatorname{cyc}}$. 
  Hence the continuity of this action is equivalent to the continuity of the map $\operatorname{Homeo}^+\mathbb{R}/\mathbb{Z}\to\operatorname{Hom}_{\mathcal{K}}(\operatorname{Sim}^n_{\operatorname{cyc}},\operatorname{Sim}^n_{\operatorname{cyc}})$, $\alpha\mapsto\alpha\times\cdots\times\alpha\mid_{\operatorname{Sim}^n_{\operatorname{cyc}}}$, which is proved in a similar way to Lemma \ref{lemma2.3}. 
  \begin{flushright}
  $\Box$
  \end{flushright}
  \section{The geometric realization of dihedral sets}  
  We first show that $\Delta D$ is the correct dihedral index category
  .  
  \begin{prop} 
  \label{prop3.1}
  The category $\Delta D$ makes the family $\{D_{n+1}\}_{n\geq 0}$ into a crossed simplicial group. 
  \end{prop}
  \begin{proof} 
  We have to check conditions 1 and 2 of Definition \ref{crossed}. 
  
  1. 
  For each $n\geq 0$, define $\mathbb{Z}_+$-functors $$\tau_n:[n]_{\text{cyc}}\to[n]_{\text{cyc}}$$ $$\omega_n:[n]_{\text{cyc}}\to[n]_{\text{cyc}}^{\operatorname{op}}$$ by $\tau_n(x)=1/(n+1)+x$ and $\omega_n(x)=-1/(n+1)-x$, respectivey. 
  Then $\tau_n$ and $\omega_n$ are isomorphisms in $\Delta D$ on $[n]_{\text{cyc}}$ of order $n+1$ and $2$, respectively, and satisfy the relation $\tau_n\omega_n=\omega_n\tau_n^{-1}$. 
  In addition, any $\phi\in\operatorname{Aut}_{\Delta D}[n]_{\text{cyc}}$ can be written as a product of $\tau_n$ and $\omega_n$. 
  Indeed, if $\phi$ is covariant, then $\phi\in\operatorname{Aut}_{\Delta C}[n]_{\text{cyc}}=\langle \tau_n\rangle$ is a power of $\tau_n$. 
  If $\phi$ is contravariant, then $\phi\circ\omega_n$ is covariant, and so a power of $\tau_n$. 
  Hence $\operatorname{Aut}_{\Delta D}[n]_{\text{cyc}}$ is generated by $\tau_n$ and $\omega_n$. 
  This means that $\operatorname{Aut}_{\Delta D}[n]_{\text{cyc}}=\langle \tau_n,\omega_n\mid \tau_n^{n+1}=\omega_n^2=1, \tau_n\omega_n=\omega_n\tau_n^{-1}\rangle$ is the dihedral group of order $2(n+1)$. 
      
  2. Let $\phi:[m]_{\text{cyc}}\to[n]_{\text{cyc}}$ be a map in $\Delta D$. 
  If $\phi$ is a covariant functor, then $\phi\in\operatorname{Hom}_{\Delta C}([m]_{\text{cyc}},[n]_{\text{cyc}})$, so that it can be uniquely written as $\phi=\psi\circ g$ with $\psi$ being a map in $\Delta$ and $g\in\operatorname{Aut}_{\Delta C}[m]_{\text{cyc}}\subset\operatorname{Aut}_{\Delta D}[m]_{\text{cyc}}$. 
  If $\phi$ is contravariant, then we can uniquely write $\phi\circ\omega_m\in\operatorname{Hom}_{\Delta C}([m]_{\text{cyc}},[n]_{\text{cyc}})$ as a composite $\phi\circ\omega_m=\psi\circ g$ with $\psi$ in $\Delta$ and $g\in\operatorname{Aut}_{\Delta C}[m]_{\text{cyc}}$. 
  Multiplication by $\omega_m$ on the right yields $\phi=\psi\circ g\circ\omega_m$ with $g\circ\omega_m\in\operatorname{Aut}_{\Delta D}[m]_{\text{cyc}}$.  
  \end{proof}

  The dihedral set $X[-]$ is extended uniquely up to unique isomorphism to $\widetilde{X}[-]:\Delta_{\operatorname{big}}D\to\operatorname{Sets}$ in the exactly same way as the procedure in subsection 2.2, except that the map $f$ can be a contravariant one. 
  The extension of the standard dihedral set $\operatorname{Hom}_{\Delta D}([-],[n]_{\operatorname{cyc}})$ is given by $\lambda\mapsto\operatorname{Hom}_{\Delta_{\operatorname{big}}D}(\lambda, [n]_{\operatorname{cyc}})$.  \\ ~ \\
  %
  %
  {\bf Proof of Theorem \ref{thm1.4}}. 
  By definition, $\abs{X[-]}$ is the realization of the underlying cyclic set $X\mid_{(\Delta C)^{\operatorname{op}}}[-]$. 
  Hence we have $$\abs{X[-]}=\abs{X\mid_{(\Delta C)^{\operatorname{op}}}[-]}=\operatorname{colim}_{F\in\mathcal{F}^\prime}\widetilde{(X\mid_{(\Delta C)^{\operatorname{op}}})}[\pi_0(\mathbb{R}/\mathbb{Z}\setminus F)]. $$ 
  The extension $\widetilde{(X\mid_{(\Delta C)^{\operatorname{op}}})}[-]$ of $X\mid_{(\Delta C)^{\operatorname{op}}}[-]$ to $(\Delta_{\text{big}}C)^{\operatorname{op}}$ is nothing but the restriction $\widetilde{X}\mid_{(\Delta_{\text{big}}C)^{\operatorname{op}}}[-]$ of $\widetilde{X}[-]$ to $(\Delta_{\text{big}}C)^{\operatorname{op}}$, 
  so that we see 
  \begin{align*}
  \abs{X[-]}&=\operatorname{colim}_{F\in\mathcal{F}^\prime}\widetilde{X}\mid_{(\Delta_{\text{big}}C)^{\operatorname{op}}}[\pi_0(\mathbb{R}/\mathbb{Z}\setminus F)]
  =\operatorname{colim}_{F\in\mathcal{F}^\prime}\widetilde{X}[\pi_0(\mathbb{R}/\mathbb{Z}\setminus F)].  
  \end{align*}  
    
  Any homeomorphism of $\mathbb{R}/\mathbb{Z}$, even an orientation-reversing one, gives rise to an isomorphism $\rho_\alpha:\abs{X[-]}\to\abs{X[-]}$ defined by $$\rho_\alpha\circ\operatorname{in}_F=\operatorname{in}_{\alpha(F)}\circ\widetilde{X}[\alpha_F^{-1}],$$ 
  where $\alpha_F:\pi_0(\mathbb{R}/\mathbb{Z}\setminus F)\to\pi_0(\mathbb{R}/\mathbb{Z}\setminus \alpha(F))$ is the covariant or contravariant $\mathbb{Z}_+$-isomorphim induced by $\alpha$.   
  Thus $\operatorname{Homeo}\mathbb{R}/\mathbb{Z}$ acts on $\abs{X[-]}$, and the proof of the continuity is as follows, being analogous to the cyclic case. 
  First we identify, by using the metric on $\abs{X[-]}$ described in the same way as in the proof of Theorem \ref{thm1.3}, the realization of the standard dihedral set $\operatorname{Hom}_{\Delta D}([-],[n]_{\operatorname{cyc}})$ with the space $\operatorname{Sim}^n_{\operatorname{dih}}$ of points $(x_0,\ldots,x_n)$ of $\mathbb{R}/\mathbb{Z}^{n+1}$ such that $x_0,\ldots,x_n$ or $x_n,\ldots,x_0$ are in the correct cyclic order. 
  Then we check that the map $\operatorname{Homeo}\mathbb{R}/\mathbb{Z}\to\operatorname{Hom}_{\mathcal{K}}(\operatorname{Sim}^n_{\operatorname{dih}},\operatorname{Sim}^n_{\operatorname{dih}})$, $\alpha\mapsto\alpha\times\cdots\times\alpha\mid_{\operatorname{Sim}^n_{\operatorname{dih}}}$ is continuous.    
  \begin{flushright}
  $\Box$
  \end{flushright}

  \section{Subdivisions} 
  \subsection{Simplicial sets and cyclic sets} 
  
  Let $X[-]$ be a simplicial set. 
  For every positive integer $r$, let $\operatorname{sd}_r:\Delta\to\Delta$ be the functor defined on objects by $\operatorname{sd}_r[n]=[r(n+1)-1]$ and on morphisms by $\operatorname{sd}_r[f](a(m+1)+b)=a(n+1)+f(b)$, where $f:[m]\to[n]$, $0\leq a<r$, and $0\leq b\leq m$. 
  The composite $\operatorname{sd}_rX[-]=X[-]\circ\operatorname{sd}_r$, which is again a simplicial set, is defined to be the $r$-fold edgewise subdivision of $X[-]$.   
  
  If $X[-]$ is a cyclic set, its subdivisions are defined analogously, but they are not cyclic sets but $\Delta_rC$-sets. 
  The category $\Delta_rC$, for each $r$, is defined to make the family $\{C_{r(n+1)}\}_{n\geq 0}$ into a crossed simplicial group, by using the $\mathbb{Z}_+$-category $\mathbb{R}/r\mathbb{Z}$ instead of $\mathbb{R}/\mathbb{Z}$. 
  Denote by $[n]_r$ the subset (considered as a $\mathbb{Z}_+$-subcategory) $\{[k+l/(n+1)]\mid0\leq k<r, 0\leq l\leq n\}$ of $\mathbb{R}/r\mathbb{Z}$. 
  We define $\Delta_rC$ 
  to have as objects the $\mathbb{Z}_+$-categories $[n]_r\subset\mathbb{R}/r\mathbb{Z}$, $n\geq 0$, 
  and to have as morphisms from $[m]_r$ to $[n]_r$, $\mathbb{Z}_+$-functors satisfying $f(x+1)=f(x)+1$. 
  The simplicial index category $\Delta$ is embedded into $\Delta_rC$ via the functor that sends $[n]$ to $[n]_r$ and $f:[m]\to[n]$ to $f_r:[m]_r\ni k+l/(m+1)\mapsto k+f(l)/(m+1)\in[n]_r$. 
  
  Let $\operatorname{sd}_r:\Delta_rC\to\Delta C$ be the functor that is defined on objects by $\operatorname{sd}_r[n]_r=[r(n+1)-1]_{\operatorname{cyc}}$ and on morphisms by $\operatorname{sd}_r[f]=\rho_n^{-1}\circ f\circ\rho_m:[r(m+1)-1]_{\operatorname{cyc}}\to[r(n+1)-1]_{\operatorname{cyc}}$, where $f:[m]_r\to[n]_r$ is a map in $\Delta_rC$ and $\rho_m$ and $\rho_n$ are the set bijections from $[m]_r$ to $[r(m+1)-1]_{\text{cyc}}$, and from $[n]_r$ to $[r(n+1)-1]_{\text{cyc}}$, respectively, induced by the isomorphism $\rho:\mathbb{R}/r\mathbb{Z}\to\mathbb{R}/\mathbb{Z}$, $x\mapsto x/r$. 
  This functor is an extension of the subdivision functor for simplicial sets, in the sense that the diagram 
  \begin{displaymath}
  \begin{CD}
  \Delta_rC @>{\operatorname{sd}_r}>> \Delta C  \\ 
  @AAA @AAA  \\
  \Delta @>{\operatorname{sd}_r}>> \Delta
  \end{CD}
  \end{displaymath}   
  commutes.  
  The $r$-fold edgewise subdivision $\operatorname{sd}_rX[-]$ of the cyclic set $X[-]$ is defined to be the composite $X[-]\circ\operatorname{sd}_r$. \\ ~ \\ 
  {\bf Proof of Theorem \ref{thm1.5}}. \\
  \underline{Case of simplicial sets}. By Drinfeld's formula for simplicial sets, we have $$\abs{\operatorname{sd}_rX[-]}=\operatorname{colim}_{F\in\mathcal{F}}\widetilde{\operatorname{sd}_rX}[\pi_0([0,1]\setminus F)].$$ 
  For each $F\in\mathcal{F}$, let $F_r$ denote the finite set $\{n+x\mid0\leq n<r, x\in F\cup\{0,1\}\}\subset[0,r]$. 
  If $n$ is the cardinality of $\pi_0([0,1]\setminus F)$, then that of $\pi_0([0,r]\setminus F_r)$ is $rn$. 
  In this case we have $\widetilde{\operatorname{sd}_rX}[\pi_0([0,1]\setminus F)]=\operatorname{sd}_rX[n-1]=X[rn-1]
  =\widetilde{X}[\pi_0([0,r]\setminus F_r)]$ 
  by construction.  
  Therefore the realization of the subdivision can be rewritten as $$\abs{\operatorname{sd}_rX[-]}=\operatorname{colim}_{F\in\mathcal{F}}\widetilde{X}[\pi_0([0,r]\setminus F_r)]. $$

  Now we compare the index categories of the colimits $\operatorname{colim}_{F\in\mathcal{F}}\widetilde{X}[\pi_0([0,r]\setminus F_r)]$ and $\operatorname{colim}_{F\in\mathcal{F}_r}\widetilde{X}[\pi_0([0,r]\setminus F)]$.  
  For every $F\in\mathcal{F}_r$ there exists a set $F^\prime\in\mathcal{F}$ such that $F\subset(F^\prime)_r$.  
  Indeed, $\{x\in[0,1]\mid n+x\in F\text{ for some }0\leq n<r\}\subset[0,1]$ is such a set.   
  This means that the subcategory of $\mathcal{F}_r$ consisting of subsets in $[0,r]$ of the form $F_r$ with $F\in\mathcal{F}$ is cofinal, whence we obtain the expression of the statement. \\ ~ \\
  %
  %
  \underline{Case of cyclic sets}. In general, it can be likewise proved that the geometric realization of a $\Delta_rC$-set $Y[-]$ is given by $$\operatorname{colim}_F\widetilde{Y}[\pi_0(\mathbb{R}/r\mathbb{Z}\setminus F)],$$ where $F$ runs through finite subsets of $\mathbb{R}/r\mathbb{Z}$ suth that $\operatorname{card}\pi_0(\mathbb{R}/r\mathbb{Z})=r(n+1)$, $n\geq 0$, and where $\widetilde{Y}[-]$ is the extension of $Y[-]$ to the category $\Delta_{r, \text{big}}C$ of $\mathbb{Z}_+$-categories isomorphic to some $[n]_r$, that has as morphisms from $\mathcal{A}$ to $\mathcal{B}$, $\mathbb{Z}_+$-functors $f$ such that there exists a map $f^\prime$ in $\Delta_rC$ such that the diagram 
  \begin{displaymath}
  \begin{CD}
  \mathcal{A} @>{\iota_{\mathcal{A}}}>> [m]_r  \\ 
  @V{f}VV @V{f^\prime}VV  \\
  \mathcal{B} @>{\iota_{\mathcal{B}}}>> [n]_r 
  \end{CD}
  \end{displaymath}  
  commutes, where $\iota_{\mathcal{A}}$ and $\iota_{\mathcal{B}}$ are chosen isomorphisms. 
  Thus $\abs{\operatorname{sd}_rX[-]}$ is the colimit $$\operatorname{colim}_F\widetilde{\operatorname{sd}_rX}[\pi_0(\mathbb{R}/r\mathbb{Z}\setminus F)], $$ which can be deformed into the desired form in a similar way to the previous case. 
  \begin{flushright}
  $\Box$
  \end{flushright}

  \subsection{Dihedral sets}
  
  Define the category $\Delta_rD$ to have the same set of objects as $\Delta_rC$, and to have as morphisms covariant $\mathbb{Z}_+$-functors $f$ satisfying $f(x+1)=f(x)+1$ and contravariant $\mathbb{Z}_+$-functors $g$ satisfying $g(x+1)=g(x)-1$. 
  This category contains as subcategories $\Delta_rC$ and, in particular, $\Delta$. 
  We notice that if we write $$\tau_{r,n}:[n]_r\to[n]_r$$ $$\omega_{r,n}:[n]_r\to[n]_r$$ for the isomorphisms in $\Delta_rD$ given by $\tau_{r,n}(x)=1/(n+1)+x$ and $\omega_{r,n}(x)=-1/(n+1)-x$, respectively, for each $n$, then they satisfy the relations $\tau_{r,n}^{r(n+1)}=\omega_{r,n}^2=1, \tau_{r,n}\omega_{r,n}=\omega_{r,n}\tau_{r,n}^{-1}$. 
  Moreover, an argument analogous to Proposition \ref{prop3.1} shows that $\Delta_rD$ is generated by $\tau_{r,n}$, $\omega_{r,n}$, and $f_r$ with $f$ in $\Delta$. 
  Therefore $\Delta_rD$ makes $\{D_{r(n+1)}\}_{n\geq 0}$ into a crossed simplicial group. 
  We also note that $\Delta_rD$ has a presentation described as follows. 
  If $d^i:[n-1]\to[n]$ and $s^i:[n+1]\to[n]$, $0\leq i\leq n$, denote the face and degeneracy operators in $\Delta$, then $\Delta_rD$ is generated by $d^i_r$, $s^i_r$, $\tau_{r,n}$, and $\omega_{r,n}$, subject to the relations: 
  \begin{description}
  \item[(S-1)] $d^i_rd^j_r=d^j_rd^{i-1}_r$ $(j<i)$ 
  \item[(S-2)] $s^i_rs^j_r=s^{j-1}_rs^i_r$ $(i<j)$ 
  \item[(S-3)] $s^i_rd^j_r=
  \begin{cases} 
  d^j_rs^{i-1}_r & (j<i)\\
  1 & (i=j, j-1)\\
  d^{j-1}_rs^i_r & (i<j-1)
  \end{cases}$  
  \item[(D-1)] $\omega_{r,n}^2=\tau_{r,n}^{r(n+1)}=1$ 
  \item[(D-2)] $\tau_{r,n}\omega_{r,n}=\omega_{r,n}\tau_{r,n}^{-1}$ 
  \item[(SD-1)] $\omega_{r,n}d^i_r=d^{n-i}_r\omega_{r,n-1}$ $(0\leq i\leq n)$
  \item[(SD-2)] $\omega_{r,n}s^i_r=s^{n-i}_r\omega_{r,n+1}$ $(0\leq i\leq n)$ 
  \item[(SD-3)] $\tau_{r,n}d^i_r=
  \begin{cases}
  d^{i+1}_r\tau_{r,n-1} & (i\neq n)\\
  d^0_r & (i=n) 
  \end{cases}$   
  \item[(SD-4)] $
  \tau_{r,n}s^i_r=
  \begin{cases}
  s^{i+1}_r\tau_{r,n+1} & (i\neq n)\\
  s^0_r\tau_{r,n+1}^2 & (i=n) 
  \end{cases}$ 
  \end{description}

  The extended category $\Delta_{r,\text{big}}D$, and the extension of a $\Delta_rD$-set $Y[-]
  $ to $\widetilde{Y}[-]:(\Delta_{r,\operatorname{big}}D)^{\operatorname{op}}\to\operatorname{Sets}$ are defined likewise. 
  The geometric realization of the $\Delta_rD$-set $Y[-]$ is given by $$\operatorname{colim}_{F}\widetilde{Y}[\pi_0(\mathbb{R}/r\mathbb{Z}\setminus F)],$$ where $F$ runs through finite subsets of $\mathbb{R}/r\mathbb{Z}$ suth that $\operatorname{card}\pi_0(\mathbb{R}/r\mathbb{Z})=r(n+1)$, $n\geq 0$. 
  
  Let $X[-]$ be a dihedral set. 
  For each $r$, the dihedral subdivision functor $\operatorname{sd}_r:\Delta_rD\to\Delta D$ is constructed in the same way as the cyclic subdivision functor. 
  \begin{df}
  We define $\operatorname{sd}_rX[-]$ to be the $\Delta_rD$-set $X[-]\circ\operatorname{sd_r}:(\Delta_rD)^{\operatorname{op}}\to\operatorname{Sets}$
  . 
  \end{df} 
  {\bf Remark.} 
  Spali\'{n}ski \cite{spalinski} defined $\operatorname{sd}_rX[-]$ to be the $r$-fold edgewise subdivision of the underlying simplicial set $X\mid_{\Delta^{\operatorname{op}}}[-]$. 
  Our definition is compatible with Spali\'{n}ski's one since we have the following commutative diagram:  
  \begin{displaymath}
  \begin{CD}
  (\Delta_rD)^{\operatorname{op}} @>{\operatorname{sd}_r}>> (\Delta D)^{\operatorname{op}} @>{X[-]}>> \operatorname{Sets} \\ 
  @AAA @AAA  @|\\
  \Delta^{\operatorname{op}}    @>{\operatorname{sd}_r}>> \Delta^{\operatorname{op}} @>{X\mid_{\Delta^{\operatorname{op}}}[-]}>> \operatorname{Sets}  
  \end{CD}
  \end{displaymath} 
  
  \subsubsection{Combination with Quillen-Segal's edgewise subdivision}
  Spali\'{n}ski \cite{spalinski} introduced another subdivision $\operatorname{sd}_r^{\operatorname{e}}X[-]$ of the dihedral set $X[-]$, for each $r\geq 1$, combining $\operatorname{sd}_r$ with Quillen-Segal's subdivision functor $\operatorname{sd}^{\operatorname{e}}$ defined in \cite{segal}. 
  The functor $\operatorname{sd}^{\operatorname{e}}:\Delta\to\Delta$ is given on objects by $\operatorname{sd}^{\operatorname{e}}[n]=[2n+1]$, and on morphisms by $\operatorname{sd}^{\operatorname{e}}[f]=f^{\operatorname{e}}$, where $f:[m]\to[n]$ is a map in $\Delta$ and $f^{\operatorname{e}}:[2m+1]\to[2n+1]$ is the map defined by $f^{\operatorname{e}}(k)=f(k)$ and $f^{\operatorname{e}}(2m+1-k)=2n+1-f(k)$ for $0\leq k\leq m$. 
  In Spali\'{n}ski's definition, $\operatorname{sd}_r^{\operatorname{e}}X[-]$ is the composite $X\mid_{\Delta^{\operatorname{op}}}[-]\circ\operatorname{sd}_r\circ\operatorname{sd}^{\operatorname{e}}$ of the underlying simplicial set of $X[-]$ with $\operatorname{sd}_r$ and $\operatorname{sd}^{\operatorname{e}}$.

  In fact, $\operatorname{sd}_r^{\operatorname{e}}X[-]$ can be defined as a $\Delta_{2r}D$-set as follows. 
  Let 
  $\operatorname{sd}^{\operatorname{e}}_r:\Delta_{2r}D\to\Delta_rD$ be the functor that is given on objects by $\operatorname{sd}^{\operatorname{e}}_r[n]_{2r}=[2n+1]_r$ and on morphisms by $\operatorname{sd}^{\operatorname{e}}_r[\tau_{2r,n}]=\tau_{r,2n+1}$, $\operatorname{sd}^{\operatorname{e}}_r[\omega_{2r,n}]=\omega_{r,2n+1}$, and $\operatorname{sd}^{\operatorname{e}}_r[f_{2r}]=f^{\operatorname{e}}_r$, where $f:[m]\to[n]$ is a map in $\Delta$ and $f^{\operatorname{e}}_r:[2m+1]_r\to[2n+1]_r$ is the map sending $l/(2(m+1))$ to $f(l)/(2(n+1))$ and $-1/(2(m+1))-l/(2(m+1))$ to $-1/(2(n+1))-f(l)/(2(n+1))$ for $0\leq l\leq m$.  
  \begin{df}
  \label{df3.1}
  We define $\operatorname{sd}_r^{\operatorname{e}}X[-]$ to be the $\Delta_{2r}D$-set $X[-]\circ\operatorname{sd}_r\circ\operatorname{sd}^{\operatorname{e}}_r:(\Delta_{2r}D)^{\operatorname{op}}\to\operatorname{Sets}$. 
  \end{df}
  {\bf Remark}. 
  The diagram 
  \begin{displaymath}
  \begin{CD}
  \Delta_{2r}D @>{\operatorname{sd}^{\operatorname{e}}_r}>> \Delta D_r \\ 
  @AAA @AAA  \\
  \Delta    @>{\operatorname{sd}^{\operatorname{e}}}>> \Delta  
  \end{CD}
  \end{displaymath} 
  commutes, and in view of this our definition of $\operatorname{sd}_r^{\operatorname{e}}X[-]$ is compatible with that of Spali\'{n}ski \cite{spalinski}.

  The proof of Theorem \ref{thm1.6} is similar to Theorem \ref{thm1.5}. 
    
  \subsubsection{Simplicial actions on subdivisions} 
  Consider the subgroup of $\operatorname{Aut}_{\Delta_
  {2r}D}[n]_
  {2r}$ generated by 
  $\tau_{2r,n}^{2(n+1)}$ and $\omega_
  {2r,n}$, which is identified with the dihedral group $D_r$. 
  Let $f:[m]\to[n]$ be a map in $\Delta$ and consider the images $\tau_{r,2n+1}^{2(n+1)}$, $\omega_{r,2n+1}$, and $f_r^{\operatorname{e}}$, of $\tau_{2r,n}^{2(n+1)}$, $\omega_{2r,n}$, and $f_{2r}$, respectively, by the functor $\operatorname{sd}_r^{\operatorname{e}}$. 
  Then the diagram 
  \begin{displaymath}
  \begin{CD}
  [2m+1]_r @>{\tau_{r,2m+1}^{2(m+1)}~\text{or}~\omega_{r,2m+1}}>> [2m+1]_r \\ 
  @V{f_r^{\operatorname{e}}}VV @V{f_r^{\operatorname{e}}}VV  \\
  [2n+1]_r @>{\tau_{r,2n+1}^{2(n+1)}~\text{or}~\omega_{r,2n+1}}>> [2n+1]_r  
  \end{CD}
  \end{displaymath}
  commutes 
  . 
  Indeed, for $0\leq l\leq m$, we have for instance $f_r^{\operatorname{e}}(\omega_{r,2m+1}(l/(2(m+1))))=f_r^{\operatorname{e}}(-1/(2(m+1))-l/(2(m+1)))=-1/(2(n+1))-f(l)/(2(n+1))$, and $\omega_{r,2n+1}(f_r^{\operatorname{e}}(l/(2(m+1)))=\omega_{r,2n+1}(f(l)/(2(n+1)))=-1/(2(n+1))-f(l)/(2(n+1))$.  
  This means that $D_r\subset\operatorname{Aut}_{\Delta_{2r}}[n]_{2r}$ acts on $\operatorname{sd}_r^{\operatorname{e}}X[-]$ simplicially, so that it is possible to define a simplicial set by $[n]\mapsto(\operatorname{sd}_r^{\operatorname{e}}X[n])^{D_r}$. 
  We also note that the action of $D_r\subset\operatorname{Aut}_{\Delta_{2r}}[n]_{2r}$ on $\abs{\operatorname{sd}_r^{\operatorname{e}}X[-]}$ is nothing but the action obtained by using the action of $\operatorname{Homeo}\mathbb{R}/2r\mathbb{Z}$ in Theorem \ref{thm1.6} and by identifying $D_r$ with the subgroup of $\operatorname{Homeo}\mathbb{R}/2r\mathbb{Z}$ generated by $\tau:x\mapsto x+2$ and $\omega:x\mapsto -x$. 

  We give a new proof to the following result of Spali\'{n}ski \cite{spalinski}: 
  \begin{prop}
  There is a canonical homeomorphism from $\abs{(\operatorname{sd}_r^{\operatorname{e}}X[-])^{D_r}}$ to $(\abs{X[-]})^{D_r}$. 
  \end{prop}
  \begin{proof}
  The left-hand-side is given by the colimit $\operatorname{colim}_{F\in\mathcal{F}}(\widetilde{\operatorname{sd}_r^{\operatorname{e}}X}[\pi_0([0,1]\setminus F)])^{D_r}$. 
  If $x\in\abs{\operatorname{sd}_r^{\operatorname{e}}X[-]}=\operatorname{colim}_{F\in\mathcal{F}}\widetilde{\operatorname{sd}_r^{\operatorname{e}}X}[\pi_0([0,1]\setminus F)]$ is represented by an element of $(\widetilde{\operatorname{sd}_r^{\operatorname{e}}X}[\pi_0([0,1]\setminus F)])^{D_r}$ with some $F\in\mathcal{F}$, then $x$ is fixed by the $D_r$-action on $\abs{\operatorname{sd}_r^{\operatorname{e}}X[-]}$. 
  The converse also holds. 
  Indeed, suppose $x\in\abs{\operatorname{sd}_r^{\operatorname{e}}X[-]}$ to be represented by $y\in\widetilde{\operatorname{sd}_r^{\operatorname{e}}X}[\pi_0([0,1]\setminus F)]$ and to be fixed by the $D_r$-action. 
  Then for any $\delta\in D_r$, there is a larger subset $G\subset[0,1]$ containing $F$ such that the images of $y$ and $\delta\cdot y$ in $\widetilde{\operatorname{sd}_r^{\operatorname{e}}X}[\pi_0([0,1]\setminus G)]$ coincides. 
  Then $x$ is represented by this common element $z\in\widetilde{\operatorname{sd}_r^{\operatorname{e}}X}[\pi_0([0,1]\setminus G)]$, and $z$ is fixed by the action of $D_r$, i.e. $z\in(\widetilde{\operatorname{sd}_r^{\operatorname{e}}X}[\pi_0([0,1]\setminus G)])^{D_r}$. 
  Therefore $\abs{(\operatorname{sd}_r^{\operatorname{e}}X[-])^{D_r}}=\operatorname{colim}_{F\in\mathcal{F}}(\widetilde{\operatorname{sd}_r^{\operatorname{e}}X}[\pi_0([0,1]\setminus F)])^{D_r}=(\operatorname{colim}_{F\in\mathcal{F}}\widetilde{\operatorname{sd}_r^{\operatorname{e}}X}[\pi_0([0,1]\setminus F)])^{D_r}=(\abs{\operatorname{sd}_r^{\operatorname{e}}X[-]})^{D_r}$. 
  Finally, the canonical homeomorphism from $\abs{\operatorname{sd}_r^{\operatorname{e}}X[-]}$ to $\abs{X[-]}$, which preserves the appropreate actions on both sides, concludes the proof.   
  \end{proof}
  
  We can also consider a simplicial action by the cyclic group $\langle \tau_{2r,n}^{2(n+1)}\rangle=C_r\subset\operatorname{Aut}_{\Delta_{2r}}[n]_{2r}$ by restricting the action by $D_r$. 
  It is proved likewise that 
  the realization of $(\operatorname{sd}_r^{\operatorname{e}}X[-])^{C_r}$ is canonically homeomorphic to that of $X[-]$. 
  Moreover, in this case the simplicial set $(\operatorname{sd}_r^{\operatorname{e}}X[-])^{C_r}$ has an extra structure. 
  Indeed, $\tau_{2r,n}^2$ and $\omega_{2r,n}$ satisfy $(\tau_{2r,n}^2)^{n+1}=\omega_{2r,n}^2=1$ and $\tau_{2r,n}^2\omega_{2r,n}=\omega_{2r,n}(\tau_{2r,n}^2)^{-1}$ on $(\operatorname{sd}_r^{\operatorname{e}}X[n])^{C_r}$. 
  Hence, $(\operatorname{sd}_r^{\operatorname{e}}X[-])^{C_r}$ is again a dihedral set.

\end{document}